\documentclass[11pt]{article}
\usepackage{amsmath,amssymb,amsthm}
\usepackage{graphicx,subfigure,float,url}
\usepackage{pdfsync}
\usepackage[english]{babel}
\usepackage[utf8]{inputenc}
\usepackage{enumitem}
\usepackage{fancybox}
\usepackage{listings}
\usepackage{mathtools}
\usepackage{amsfonts}
\usepackage[bottom]{footmisc}
\usepackage{verbatim}

\usepackage{float}
\usepackage{makeidx}
\normalbaselines

\newtheorem{Theorem}{Theorem}[section]

\newtheorem{Lemma}[Theorem]{Lemma}

\newtheorem{Corollary}[Theorem]{Corollary}

\newtheorem{Remark}[Theorem]{Remark}

\newcommand{\RR}{{{\rm I} \kern -.15em {\rm R} }}

\newcommand{\C}{{{\rm l} \kern -.42em {\rm C} }}

\newcommand{\nat}{{{\rm I} \kern -.15em {\rm N} }}

\newcommand{\be}{\begin{equation}}
\newcommand{\ee}{\end{equation}}
\newcommand{\beq}{\begin{eqnarray}}
\newcommand{\eeq}{\end{eqnarray}}
\newcommand{\beqs}{\begin{eqnarray*}}
\newcommand{\eeqs}{\end{eqnarray*}}
\newcommand{\bt}{\begin{Theorem}}
\newcommand{\et}{\end{Theorem}}
\newcommand{\br}{\begin{Remark}}
\newcommand{\er}{\end{Remark}}
\newcommand{\bc}{\begin{Corollary}}
\newcommand{\ec}{\end{Corollary}}
\newcommand{\bl}{\begin{Lemma}}
\newcommand{\el}{\end{Lemma}}
\newcommand{\bd}{\begin{definition}}
\newcommand{\ed}{\end{definition}}

\renewcommand{\geq}{\geqslant}
\renewcommand{\ge}{\geqslant}
\renewcommand{\leq}{\leqslant}
\renewcommand{\le}{\leqslant}

\title{Exponential decay for semilinear wave equations with viscoelastic damping and delay feedback}
\author{
Alessandro Paolucci\footnote{Dipartimento di Ingegneria e Scienze dell'Informazione e Matematica, Universit\`{a} di L'Aquila, Via Vetoio, Loc. Coppito, 67010 L'Aquila Italy (\texttt{alessandro.paolucci2@graduate.univaq.it}).}
\and
Cristina Pignotti\footnote{Dipartimento di Ingegneria e Scienze dell'Informazione e Matematica, Universit\`{a} di L'Aquila, Via Vetoio, Loc. Coppito, 67010 L'Aquila Italy (\texttt{pignotti@univaq.it}).}
}

\date{}

\begin{document}

\textwidth=160 mm

\textheight=225mm

\parindent=8mm

\frenchspacing

\maketitle

\begin{abstract}
In this paper we study a class of semilinear wave type equations with viscoelastic damping and delay feedback with time variable coefficient.
By combining semigroup arguments, careful energy estimates and an iterative approach we are able to prove, under suitable assumptions, a well-posedness result and an exponential decay estimate for solutions corresponding to {\em small} initial data. This extends and concludes the analysis initiated in \cite{JEE15} and then developed in \cite{KP,JEE18}.
\end{abstract}

\vspace{5 mm}

\def\qed{\hbox{\hskip 6pt\vrule width6pt
height7pt
depth1pt  \hskip1pt}\bigskip}

%% {\bf 2000 Mathematics Subject Classification:}
%%35L05, 93D15

 %%{\bf Keywords and Phrases:}  wave equation,  delay feedbacks, stabilization

\section{Introduction}
\label{pbform}
%\hspace{5mm}

\setcounter{equation}{0}

Let $H$ be a Hilbert space and let $A$ be a positive self-adjoint operator with dense domain $D(A)$ in $H$ and compact inverse in $H$. Let us consider the system:
\begin{equation}
\label{equazione_generale}
\begin{array}{l}
\displaystyle{ u_{tt}(t)+Au(t)- \int_0^{+\infty} \mu (s)Au(t-s) ds+k(t)BB^*u_t(t-\tau)=\nabla \psi( u(t)), \ \ t\in (0,+\infty)}\\
%\hspace{3 cm}
\displaystyle{
 u(t)=u_0(t), \quad t\in (-\infty, 0],}\\
%\hspace{3 cm}\\
\displaystyle{ u_t(0)=u_1}\\
\displaystyle{ B^*u_t(t)=g(t), \quad t\in (-\tau,0),}
\end{array}
\end{equation}
where $\tau>0$ represents the time delay, $B$ is a bounded linear operator of $H$ into itself, $B^*$ denotes its adjoint, and $(u_0(\cdot), u_1, g(\cdot))$ are the initial data taken in suitable spaces.
Moreover, the delay damping coefficient  $k:[0,+\infty)\rightarrow \RR$ is a function in $L^1_{loc} ([0,+\infty))$ such that
\begin{equation}\label{Cstar}
\int_{t-\tau}^t |k(s)| ds < C^*, \quad \forall t\in (0,+\infty),
\end{equation}
for a suitable constant $C^*,$
and  the memory kernel $\mu:[0,+\infty) \rightarrow [0,+\infty)$ satisfies the following assumptions:
\begin{enumerate}[label=(\roman*)]
\item $\mu \in C^1(\RR^+) \cap L^1(\RR^+)$;
\item $\mu(0)=\mu_{0}>0$;
\item $\int_0^{+\infty} \mu(t)dt=\tilde{\mu}<1$;
\item $\mu'(t)\leq -\delta \mu(t)$, for some $\delta>0$.
\end{enumerate}

Furthermore, $\psi : D(A^{\frac 1 2})\rightarrow \RR$ is a functional having G\^{a}teaux derivative $D\psi(u)$ at every $u\in D(A^{\frac 12}).$ Moreover, in the spirit of  \cite{ACS}, we assume the following hypotheses:
\begin{itemize}
\item[{(H1)}] For every $u\in D(A^{\frac 1 2})$, there exists a constant $c(u)>0$ such that
$$
|D\psi(u)(v)|\leq c(u) ||v||_{{H}} \qquad \forall v\in {D}(A^{\frac 1 2}).
$$
Then, $\psi$ can be extended to the whole  $H$ and
 we denote by $\nabla \psi(u)$ the unique vector representing $D\psi(u)$ in the Riesz isomorphism, i.e.
$$
\langle \nabla \psi(u), v \rangle_H =D\psi(u) (v), \qquad \forall v\in H;
$$
\item[ (H2)] for all $r>0$ there exists a constant $L(r)>0$ such that
$$
||\nabla \psi (u)-\nabla \psi (v)||_H \leq L(r) ||A^{\frac 12}(u-v)||_H,
$$
for all $u,v\in {D}(A^{\frac 12})$ satisfying $||A^{\frac 12} u||_H\leq r$ and $||A^{\frac 12} v||_H\leq r$.
\item[{ (H3)}] $\psi(0)=0,$  $\nabla \psi(0)=0$ and
there exists a strictly increasing continuous function $h$ such that
\begin{equation}
\label{stima_h}
||\nabla \psi (u)||_H\leq h(||A^{\frac 12} u||_H)||A^{\frac 12}u||_H,
\end{equation}
for all $u\in {D}(A^{\frac 12})$.
\end{itemize}
We are interested in studying well--posedness and stability results, for small initial data, for the above model.
Our results extend the ones of \cite{JEE15, JEE18} where abstract evolution equations are analyzed and, in the specific case of memory damping, exponential decay is obtained essentially only in the linear case. Indeed, in the nonlinear setting, an extra standard frictional damping, not delayed, was needed in order to obtain existence and uniqueness of global solutions with exponentially decaying energy for suitably small initial data. Moreover, in \cite{JEE15, JEE18} the delay damping coefficient $k(t)$ is assumed to be constant and the results there obtained require a smallness assumption on $\Vert k\Vert_\infty.$ The analysis of \cite{JEE15, JEE18} has been extended in \cite{KP} by considering a time variable delay damping coefficient $k(t)$ as in the present paper. However, also in \cite{KP} an extra frictional not delayed damping was needed, in the case of wave type equation with memory damping, when a locally Lipschitz continuous nonlinear term is included into the equation.

Then, here, we focus on wave type equations with viscoelastic damping, delay feedback and source term,  obtaining well-posedness and stability results for small initial data without adding extra frictional not delayed damping. So, we here improve and conclude the analysis developed in \cite{JEE15, JEE18, KP} for the class of models at hand.
Other models with viscoelastic damping and time delay are studied in recent literature. The first result is due to \cite{KSH}, in the linear setting. In that paper a standard frictional damping, not delayed, is included into the model in order to compensate the destabilizing effect of the delay feedback. Actually, at least in the linear case, the viscoelastic damping alone can counter the destabilizing delay effect, under suitable assumptions, without needing other dampings. This has been shown, e.g., in \cite{Guesmia, ANP, Dai, Yang}. The case of viscoelastic  wave equation with intermittent delay feedback has been studied in \cite{P} while \cite{MK} deals with a model for plate equation with memory, source term, delay feedback and standard not delayed frictional damping.

More extended is the literature in case of a frictional/structural damping, instead of a viscoelastic term, which compensates the destabilizing effect of a time delay and, for specific models, mainly in the linear setting, several stability results have been quite recently obtained under appropriate assumptions (see e.g. \cite{AABM, AG, AM, C, NP, OO, SS, XYL}).

The paper is organized as follows. In Section 2 we give some preliminaries, writing system \eqref{equazione_generale} in an abstract way. In Section 3 we prove the exponential decay of the energy associated to \eqref{equazione_generale}. Finally, in Section 4 some examples are illustrated.

\section{Preliminaries}

As in Dafermos \cite{Dafermos}, we define the function
\begin{equation}
\label{eta}
\eta^t(s):=u(t)-u(t-s),\quad s,t\in (0,+\infty),
\end{equation}
so that we can rewrite \eqref{equazione_generale} in the following way:
\begin{equation}
\label{equazione_riscritta}
\begin{array}{l}
\displaystyle{u_{tt}(t)+(1-\tilde{\mu})Au(t)+\int_0^{+\infty} \mu (s)A\eta^t(s)ds+k(t)BB^*u_t(t-\tau)}\\
\displaystyle{\hspace{7 cm}
=\nabla \psi(u(t)),\ t\in(0,+\infty),}\\
\displaystyle{ \eta^t_t(s)=-\eta^t_s(s)+u_t(t),\ \quad t, s\in  (0, +\infty),}\\
%%\displaystyle{ u(x,t)=0 \quad \text{in} \quad \partial \Omega \times (0,+\infty),}\\
%%\displaystyle{ \eta^t(x,s)=0 \quad \text{in} \quad \partial\Omega \times (0,+\infty),}\\
\displaystyle{ u(0)=u_0(0),}\\
\displaystyle{ u_t(0)=u_1,}\\
\displaystyle{B^*u_t(t)= g(t), \quad t\in (-\tau, 0),}\\
\displaystyle{ \eta^0(s)=\eta_0(s)=u_0(0)-u_0(-s)  \quad s\in (0,+\infty).}
\end{array}
\end{equation}

Let $L^2_\mu ((0,+\infty); D(A^{\frac 12}))$ be the Hilbert space of the $D(A^{\frac 12})-$valued functions in $(0,+\infty)$ endowed with the scalar product

$$\langle \varphi, \psi \rangle_{L^2_\mu ((0,+\infty); D(A^{\frac 12}))}=\int_0^\infty \mu (s)\langle A^{\frac 12}\varphi, A^{\frac 12}\psi \rangle_H ds$$
and denote by ${\mathcal H }$ the Hilbert space
$$
\mathcal{H}=D(A^{\frac 12})\times H\times  L^2_{\mu} ((0,+\infty);D(A^{\frac 12})),
$$
equipped with the inner product

\begin{equation}\label{inner}
\left\langle
\left (
\begin{array}{l}
u\\
v\\
w
\end{array}
\right ),
\left (
\begin{array}{l}
\tilde u\\
\tilde v\\
\tilde w
\end{array}
\right )
\right\rangle_{\mathcal{H}}:= (1-\tilde\mu ) \langle A^{\frac 12}u, A^{\frac 12} \tilde u\rangle_H+\langle v, \tilde v\rangle_H+\int_0^\infty \mu (s)\langle  A^{\frac 12}w,  A^{\frac 12} \tilde w\rangle_H ds.
\end{equation}
Setting $U= (u, u_t, \eta^t)$ we can restate \eqref{equazione_generale} in the abstract form
\begin{equation}
\label{forma_astratta2}
\begin{array}{l}
\displaystyle{ U'(t)= \mathcal{A} U(t)-k(t)\mathcal{B}U(t-\tau)+F(U(t)),}\\
\displaystyle{ \mathcal{B}U(t-\tau)=\tilde g(t) \quad \text{for} \quad t\in [0,\tau],}\\
\displaystyle{ U(0)=U_0,}
\end{array}
\end{equation}
where the operator ${\mathcal A}$  is defined by
\begin{equation*}
\mathcal{A} \begin{pmatrix}
u\\
v\\
w
\end{pmatrix}
=
\begin{pmatrix}
v\\
-(1-\tilde{\mu})Au-\int_0^{+\infty} \mu(s) A w(s) ds \\
-w_s +v
\end{pmatrix}
\end{equation*}
with domain
\begin{equation}
\begin{array}{c}
\displaystyle{ {D}(\mathcal{A})= \{ (u,v,w) \in D(A^{\frac 12})\times D(A^{\frac 12})\times L_{\mu}^2 ((0,+\infty); D(A^{\frac 12})):}\hspace{1,5 cm}\\
\hspace{1.5 cm}
\displaystyle{(1-\tilde{\mu})u+\int_0^{+\infty} \mu (s) w(s)ds \in D(A), \quad w_s \in L^2_{\mu} ((0,+\infty); D(A^{\frac 12}))\},}
\end{array}
\end{equation}
in the Hilbert space ${\mathcal H},$ and the operator ${\mathcal B}:{\mathcal H}\rightarrow {\mathcal H}$ is defined by
$${\mathcal B}\left (
\begin{array}{l}
u\\
v\\
w
\end{array}
\right ):= \left(
\begin{array}{l}
0\\
 BB^* v\\
 0
\end{array}
\right ).$$
Moreover,  $\tilde g(t)=(0, Bg(t-\tau), 0),$ $U_0=(u_0(0), u_1, \eta_0)$ and  $F(U):= (0, \nabla\psi(u), 0)^T.$ From (H2) and (H3) we deduce that
the function $F$ satisfies:
\begin{itemize}
\item[{(F1)}] $F(0)=0$;
\item[{(F2)}] for each $r>0$ there exists a constant $L(r)>0$ such that
\begin{equation}
||F(U)-F(V)||_{\mathcal{H}}\leq L(r)||U-V||_{\mathcal{H}}
\end{equation}
whenever $||U||_{\mathcal{H}}\leq r$ and $||V||_{\mathcal{H}}\leq r$.
\end{itemize}

\section{Stability result}
In this section we want to prove an exponential stability result for the system \eqref{equazione_generale} for {\em small} initial data.
It's well--known (see e.g. \cite{Giorgi}) that the operator ${\mathcal A}$
in the problem's formulation \eqref{forma_astratta2} generates an exponentially stable semigroup
$\{S(t)\}_{t\geq0},$ namely there exist two costants
$M,\omega >0$ such that
\begin{equation}
\label{semigruppo}
||S(t)||_{\mathcal{L}(H)}\leq Me^{-\omega t}.
\end{equation}
Denoting
\begin{equation}\label{b}
\Vert B\Vert_{\mathcal{L}(H)}=\Vert B^*\Vert_{\mathcal{L}(H)}=b,
\end{equation}
then $\Vert {\mathcal B}\Vert_{\mathcal{L(H)}} =b^2.$
Our result will be obtained under an assumption on the coefficient $k(t)$ of the delay feedback which includes as particular cases $k$ integrable and $k$ in $L^\infty$ with $\Vert k\Vert_\infty$ sufficiently small.
More precisely, we assume (cf. \cite{KP}) that
there exist two constants $\omega '\in [0,\omega)$ and $\gamma\in \RR$ such that
\begin{equation}
\label{ipotesi2}
b^2Me^{\omega\tau} \int_0^t  |k(s+\tau)| ds \leq \gamma +\omega 't, \quad \mbox{\rm for all}\ t\ge 0.
\end{equation}

\begin{Theorem}\label{generaleCV}
Assume \eqref{ipotesi2}. Moreover, suppose that
\begin{itemize}
\item[{(W)}] there exist $\rho>0$, $C_\rho>0$,  with $L(C_\rho)<\frac{\omega-\omega '}{M}$ such that if $U_0 \in \mathcal{H}$ and if $\tilde g\in C([0,\tau];\mathcal{H})$ satisfy
\begin{equation}\label{well-posedness}
||U_0||^2_{\mathcal{H}}+\int_0^\tau |k(s)| \cdot ||\tilde{g}(s)||^2_{\mathcal{H}} ds <\rho ^2,
\end{equation}
then the system \eqref{forma_astratta2} has a unique solution $U\in C([0,+\infty);\mathcal{H})$ satisfying $||U(t)||_{\mathcal{H}}\leq C_\rho$ for all $t>0$.
\end{itemize}
Then, for every solution $U$ of \eqref{forma_astratta2}, with initial datum $U_0$ satisfying \eqref{well-posedness},
\begin{equation}
\label{stimaesponenziale}
||U(t)||_{\mathcal H}\leq \tilde{M}\left (||U_0||_{\mathcal H}+\int_0^\tau e^{\omega s } |k(s)|\cdot ||\tilde g(s)||_{\mathcal{H}} ds\right )e^{-(\omega -\omega'-ML(C_{\rho}))t}, \quad t\ge 0,
\end{equation}
with $\tilde M= M e^{\gamma}.$
\end{Theorem}

\begin{proof}
By Duhamel's formula, using \eqref{semigruppo}, we have
\begin{equation*}
\begin{array}{l}
\displaystyle{ ||U(t)||_{\mathcal H}\leq Me^{-\omega t}||U_0||_{\mathcal H}+Me^{-\omega t} \int_0^t e^{\omega s} |k(s)|\cdot ||\mathcal{B}U(s-\tau)||_{\mathcal H}ds +ML(C_{\rho})e^{-\omega t} \int_0^t e^{\omega s} ||U(s)||_{\mathcal H} ds } \\
\hspace{1.5 cm}
\displaystyle{ \leq Me^{-\omega t}||U_0||_{\mathcal H}+Me^{-\omega t} \int_0^\tau e^{\omega s} |k(s)|\cdot ||\mathcal{B} U(s-\tau)||_{\mathcal H} ds }\\
\hspace{1.7 cm}
\displaystyle{ + Me^{-\omega t} \int_\tau ^t e^{\omega s} |k(s)|\cdot ||\mathcal{B} U(s-\tau)||_{\mathcal H} ds + ML(C_{\rho})e^{-\omega t} \int_0^t e^{\omega s} ||U(s)||_{\mathcal H} ds. }
\end{array}
\end{equation*}
Hence, we obtain
$$
\begin{array}{l}
\displaystyle{ e^{\omega t} ||U(t)||_{\mathcal H} \leq
M||U_0||_{\mathcal H}+M\int_0^\tau e^{\omega s } |k(s)|\cdot ||{\mathcal{B}}U(s-\tau )||_{\mathcal H} ds }\\
\hspace{2 cm}
\displaystyle{ +\int_0^t \Big ( Me^{\omega \tau} |k(s+\tau)| \cdot ||\mathcal{B}||_{{\mathcal{L}}(H)} +ML(C_{\rho}) \Big )
e^{\omega s}
||U(s)||_{\mathcal H} ds,}
\end{array}
$$
and then
$$
\begin{array}{l}
\displaystyle{ e^{\omega t} ||U(t)||_{\mathcal H} \leq M||U_0||_{\mathcal H}+M\int_0^\tau e^{\omega s } |k(s)|\cdot ||\tilde g(s)||_{\mathcal H} ds }\\
\hspace{2 cm}
\displaystyle{ +\int_0^t \Big ( M b^2 e^{\omega \tau} |k(s+\tau)| +ML(C_{\rho}) \Big ) e^{\omega s}||U(s)||_{\mathcal H} ds.}
\end{array}
$$
Therefore, using Gronwall's inequality,
$$
e^{\omega t}
\Vert U(t)\Vert_{\mathcal H}\le
M\left (
\Vert U_0\Vert_{\mathcal H} +\int_0^{\tau} e^{\omega s } \vert k(s)\vert\cdot \Vert \tilde g(s)\Vert_{\mathcal H} ds
\right )
 e^{
M b^2 e^{\omega\tau}\int_0^t \vert k(s+\tau )\vert ds + ML(C_\rho )t}
$$
and so, from \eqref{ipotesi2},
$$
||U(t)||_{\mathcal H} \leq Me^{\gamma}\left (
||U_0||_{\mathcal H}+\int_0^\tau
e^{\omega s } |k(s)|\cdot ||\tilde g(s)||_{\mathcal H} ds
\right ) e^{- (\omega-\omega^\prime - ML(C_{\rho} ))t}.
$$
This gives  \eqref{stimaesponenziale} with $\tilde M$ as in the statement.
\end{proof}

In order to prove the stability result we need then to show that the well-posedness assumption (W) of Theorem \ref{generaleCV} is satisfied for problem
 \eqref{equazione_generale}.
For this, let us define  the energy of the model \eqref{equazione_generale} as
\begin{equation}
\label{energia}
\begin{array}{l}
\displaystyle{E(t):=E (u(t))=\frac{1}{2}||u_t(t)||_H^2+\frac{1-\tilde{\mu}}{2}||A^{\frac 12}u(t)||^2_H-\psi(u)  }\\ \hspace{2 cm}
\displaystyle{ +\frac{1}{2}\int_{t-\tau}^t |k(s+\tau)|\cdot ||B^*u_t(s)||_H^2 ds+\frac{1}{2}\int_0^{+\infty} \mu(s) ||A^{\frac 1 2} \eta^t(s)||^2_H ds.}
 \end{array}
\end{equation}

The following lemma holds.
\begin{Lemma}\label{stimaE}
Let $u:[0,T)\rightarrow \RR$ be a solution of \eqref{equazione_generale}.
Assume that $E(t)\geq \frac{1}{4}||u_t(t)||_H^2$ for all $t\in [0,T)$. Then,
\begin{equation}
\label{disuguaglianza energia}
E(t)\leq \bar{C}(t)E(0),
\end{equation}
for all $t\in [0,T)$, where
\begin{equation}\label{Cbar}
\bar{C}(t)=e^{2b^2\int_0^t ( |k(s)|+|k(s+\tau)|)  ds}.
\end{equation}
\end{Lemma}
\begin{proof}
Differentiating $E(t)$, we obtain
$$
\begin{array}{l}

\displaystyle{\frac{dE}{dt}=\langle u_t,u_{tt}\rangle_H+(1-\tilde{\mu})\langle A^{\frac 1 2}u,A^{\frac 1 2} u_t \rangle_H -\langle \nabla \psi (u), u_t\rangle_H +\frac{1}{2} |k(t+\tau)|\cdot ||B^*u_t(t)||_H^2 }\\
\hspace{1.5 cm}
\displaystyle{-\frac{1}{2}|k(t)|\cdot ||B^*u_t(t-\tau)||_H^2+\int_0^{+\infty} \mu (s) \langle A^{\frac 12}\eta^t(s),A^{\frac 1 2}\eta^t_t(s) \rangle_H ds.}

\end{array}
$$
Then, from \eqref{equazione_generale},
$$
\begin{array}{l}
\displaystyle{\frac{dE}{dt}= -\int_0^{+\infty} \mu(s)\langle u_t(t),A\eta^t(s)\rangle_H ds -k(t)\langle u_t, BB^*u_t(t-\tau) \rangle_H +\frac{1}{2}|k(t+\tau)|\cdot||B^*u_t(t)||^2}\\
\hspace{2 cm}
\displaystyle{-\frac{1}{2}|k(t)|\cdot||B^*u_t(t-\tau)||^2+\int_0^{+\infty} \mu (s) \langle A\eta^t(s),\eta^t_t(s)\rangle ds.}
\end{array}
$$
Using the second equation of  \eqref{equazione_riscritta}, we have that
$$
\begin{array}{l}
\displaystyle{ \frac{dE}{dt}= -k(t)\langle u_t, BB^* u_t(t-\tau)\rangle _H+\frac{1}{2}|k(t+\tau)|\cdot ||B^*u_t||_H^2-\frac{1}{2}|k(t)|\cdot ||B^*u_t(t-\tau)||_H^2}\\
\hspace{3 cm}
\displaystyle{ -\int_0^{+\infty} \mu(s) \langle A\eta^t(s),\eta _s^t (s) \rangle_H ds.}
\end{array}
$$
Now, we claim that
$$
\int_0^{+\infty} \mu(s) \langle \eta ^t_s, A\eta^t(s) \rangle_H ds \ge 0.
$$
Indeed, integrating by parts and recalling assumption (iv) on $\mu(\cdot),$ we deduce
$$
\int_0^{+\infty} \mu(s) \langle \eta ^t_s, A\eta^t(s) \rangle_H ds = -\frac{1}{2}\int_0^{+\infty} \mu'(s) ||A^{\frac{1}{2}} \eta^t(s)||_H^2 ds\ge 0.$$
Therefore, we have that
$$
\begin{array}{l}
\displaystyle{\frac{dE(t)}{dt} \leq -k(t)\langle B^*u_t,B^*u_t(t-\tau)\rangle_H +\frac{1}{2}|k(t+\tau)| \cdot ||B^*u_t(t)||_H^2-\frac{1}{2}|k(t)|\cdot ||B^*u_t(t-\tau)||_H^2}.
\end{array}
$$
Now, using Cauchy-Schwarz inequality, we obtain the following estimate:
$$
\begin{array}{l}
\displaystyle{ \frac{dE(t)}{dt}\leq \frac{1}{2}(|k(t)|+|k(t+\tau)|)||B^*u_t(t)||_H^2}\\
\hspace{1 cm}
\displaystyle{ \leq \frac{1}{2} (|k(t)|+|k(t+\tau)|) b^2 ||u_t||_H^2}\\
\hspace{1 cm}
\displaystyle{ =2b^2(|k(t)|+|k(t+\tau)|) \frac{1}{4}||u_t||_H^2}\\
\hspace{1 cm}
\displaystyle{ \le 2b^2(|k(t)|+|k(t+\tau)|) E(t)}.
\end{array}
$$
Hence, the Gronwall Lemma
concludes the proof.
\end{proof}
Before proving the well-posedness assumption (W) for solutions to \eqref{forma_astratta2}, we need the following two lemmas.
\begin{Lemma}
\label{lemma1}
Let us consider the system \eqref{forma_astratta2} with initial data $U_0\in \mathcal{H}$ and  $\tilde g\in C([0,\tau]; \mathcal{H}).$ Then, there exists a unique local solution $U(\cdot)$ defined on a time  interval $[0,\delta)$, with $\delta \le\tau$.
\end{Lemma}
\begin{proof}
Since $t\in [0,\tau]$, we can rewrite the abstract system \eqref{forma_astratta2} as an undelayed problem:
\begin{eqnarray*}
U'(t)&=& \mathcal{A}U(t)-k(t)\tilde{g}(t)+F(U(t)), \quad t\in (0, \tau),\\
U(0)&=&U_0.
\end{eqnarray*}
Then, we can
apply the classical theory of nonlinear semigroups (see e.g. \cite{Pazy}) obtaining the existence of a unique solution  on a set $[0,\delta)$, with $\delta \le\tau$.
\end{proof}
\begin{Lemma}
\label{Lemma 2}
Let $U(t)=(u(t),u_t(t),\eta^t)$ be a solution to \eqref{forma_astratta2} defined on the interval $[0, \delta),$ with $\delta\le\tau.$ Then,
\begin{enumerate}
\item if $h(||A^\frac{1}{2} u_0(0)||_H)<\frac{1-\tilde{\mu}}{2},$ then $E(0)>0$;
\item if $h(||A^\frac{1}{2} u_0(0)||_H)<\frac{1-\tilde{\mu}}{2}$ and
$
h \left( \frac{2}{(1-\tilde{\mu})^\frac{1}{2}} \bar{C}^\frac{1}{2}(\tau) E^\frac{1}{2}(0) \right) <\frac{1-\tilde{\mu}}{2},
$
with $\bar{C}(\tau)$ defined in  \eqref{Cbar}, then
\begin{equation}
\label{stima E dal basso}
\begin{array}{l}
\displaystyle{ E(t)>\frac{1}{4}||u_t||_H^2+\frac{1-\tilde{\mu}}{4}||A^\frac{1}{2}u||_H^2}\\
\hspace{1 cm}
\displaystyle{ +\frac{1}{4}\int_{t-\tau}^t |k(s+\tau)| \cdot ||B^*u_t(s)||_H^2 ds+\frac{1}{4}\int_0^{+\infty} \mu(s) ||A^\frac{1}{2}\eta^t (s)||_H^2 ds,}
\end{array}
\end{equation}
for all $t\in[0, \delta)$. In particular,
\begin{equation}\label{J2}
E(t)>  \frac 14 \Vert U(t)\Vert_{\mathcal H}^2, \quad \mbox{for all} \ \ t\in [0, \delta).
\end{equation}
\end{enumerate}
\end{Lemma}
\begin{proof}
We first deduce by assumption (H3) on $\psi$ that
\begin{equation}
\label{assumptionPsi}
\begin{array}{l}
\displaystyle{|\psi(u)|\leq \int_0^1 |\langle \nabla \psi (su),u\rangle | ds} \\
\hspace{1,15 cm}
\displaystyle{\leq  ||A^\frac{1}{2}u||^2_H \int_0^1 h(s||A^\frac{1}{2}u||_H)s ds\leq \frac{1}{2} h(||A^\frac{1}{2}u||_H)||A^\frac{1}{2}u||^2_H.}
\end{array}
\end{equation}
Hence, under the assumption $h (\Vert A^{\frac 12} u_0(0)\Vert_H) < \frac {1-\tilde \mu} 2,$ we have that
$$
\begin{array}{l}
\displaystyle{ E(0)=\frac{1}{2}||u_1||_H^2+\frac{1-\tilde{\mu}}{2}||A^\frac{1}{2}u_0(0)||_H^2-\psi(u_0(0))+\frac{1}{2}\int_{-\tau}^0 |k(s+\tau)|\cdot ||B^*u_t(s)||^2_H ds}\\
\hspace{1 cm}
\displaystyle{+\frac{1}{2}\int_0^{+\infty} \mu(s)||A^\frac{1}{2}\eta_0(s)||_H^2 ds}\\
\hspace{0,9 cm}
\displaystyle{ \ge\frac{1}{2}||u_1||^2_H+\frac{1-\tilde{\mu}}{4}||A^\frac{1}{2}u_0(0)||^2_H +\frac{1}{2}\int_{-\tau}^0 |k(s+\tau)| \cdot ||B^*u_t(s)||^2_H ds }\\
\hspace{1 cm}
\displaystyle{+\frac{1}{2}\int_0^{+\infty} \mu(s) ||A^\frac{1}{2} \eta_0(s)||^2_H ds>0,}
\end{array}
$$
obtaining $1.$\\
In order to prove the second statement, we argue by contradiction. Let us denote
$$
r:=\sup \{ s\in [0,\delta) : \eqref{stima E dal basso} \quad \text{holds} \quad \forall t\in [0,s)\}.
$$
We suppose by contradiction that $r<\delta$. Then, by continuity, we have
\begin{equation}
\label{continuita}
\begin{array}{l}
\displaystyle{E(r)=\frac{1}{4}||u_t(r)||^2_H+\frac{1-\tilde{\mu}}{4}||A^\frac{1}{2}u(r)||_H^2+\frac{1}{4}\int_{r-\tau}^r |k(s+\tau)| \cdot ||B^*u_t(s)||_H^2 ds}\\
\hspace{2 cm}
\displaystyle{ +\frac{1}{4} \int_0^{+\infty} \mu(s)||A^\frac{1}{2}\eta^r(s)||_H^2 ds.}
\end{array}
\end{equation}
Now, from \eqref{continuita} and Lemma \ref{stimaE} we can infer that
$$
\begin{array}{l}
\displaystyle{ h(||A^\frac{1}{2}u(r)||_H)\leq h\left( \frac{2}{(1-\tilde{\mu})^\frac{1}{2}} E^\frac{1}{2}(r)\right) }\\
\hspace{2,3 cm}
\displaystyle{ <h\left( \frac{2}{(1-\tilde{\mu})^\frac{1}{2}}\bar{C}^\frac{1}{2}(\tau)E^\frac{1}{2}(0)\right) <\frac{1-\tilde{\mu}}{2}.}
\end{array}
$$
Hence, we have that
$$
\begin{array}{l}
\displaystyle{ E(r)=
\frac{1}{2}||u_t(r)||_H^2+\frac{1-\tilde{\mu}}{2}||A^\frac{1}{2}u(r)||_H^2-\psi(u(r))+\frac{1}{2}\int_{r-\tau}^r|k(s+\tau)|\cdot ||B^*u_t(s)||^2_H ds}\\
\hspace{3 cm}
\displaystyle { +\frac{1}{2}\int_0^{+\infty} \mu(s) ||A^\frac{1}{2} \eta^r(s)||_H^2 ds}\\
\hspace{1 cm} \displaystyle{
>\frac{1}{4}||u_t(r)||_H^2+\frac{1-\tilde{\mu}}{4}||A^\frac{1}{2}u(r)||_H^2+\frac{1}{4}\int_{r-\tau}^r|k(s+\tau)| \cdot ||B^*u_t(s)||_H^2 ds}\\
\hspace{3 cm}
\displaystyle{ +\frac{1}{4}\int_0^{+\infty} \mu(s) ||A^\frac{1}{2} \eta^r(s)||_H^2 ds,}
\end{array}
$$
which contradicts the maximality of $r$. Hence, $r=\delta$ and this concludes the proof.
\end{proof}
\begin{Theorem}\label{teorema2.5}
Problem \eqref{forma_astratta2}, with initial data $U_0\in \mathcal{H}$ and  $\tilde g\in C([0,\tau]; \mathcal{H}),$  satisfies the well-posedness assumption (W). Then, for solutions of \eqref{forma_astratta2} corresponding to sufficiently small initial data the exponential decay estimate \eqref{stimaesponenziale} holds.
\end{Theorem}
\begin{proof}
Let us fix $N\in\nat$  such that
\begin{equation}
\label{stimaN1}
2M^2\Bigl( 1+e^{2\omega\tau}C^*\Bigr)e^{2\gamma}e^{-(\omega -\omega ')(N-1)\tau} < \frac{1}{1+e^{\omega\tau } b^2 C^*},
\end{equation}
where $C^*$ is the constant defined in \eqref{Cstar}.
Then, let $\rho$ be a positive constant such that
\begin{equation}
\label{rho}
\rho \leq \frac{(1-\tilde{\mu})^\frac{1}{2}}{2\bar{C}^\frac{1}{2}(N\tau)}h^{-1} \left(\frac{1-\tilde{\mu}}{2}\right).
\end{equation}
Now, let us assume that the initial data $(u_0(0), u_1, \eta_0)$ and $B^*u_t(s),\ s\in [-\tau, 0],$ satisfy the smallness assumption
\begin{equation}
\label{smallness_condition}
\begin{array}{l}
\displaystyle{(1-\tilde{\mu})||A^\frac{1}{2}u_0(0)||_H^2+||u_1||_H^2+\int_{-\tau}^0 |k(s+\tau)| ||B^*u_t(s)||_H^2 ds}\\
\hspace{2 cm}
\displaystyle{+
\int_0^{+\infty} \mu(s)||A^\frac{1}{2}\eta_0(s)||_H^2 ds<\rho^2.}
\end{array}
\end{equation}
Note that \eqref{smallness_condition} is equivalent to
\begin{equation}\label{condiz_gen}
\Vert U_0\Vert_{\mathcal H}^2 +\int_0^\tau \vert k(s)\vert \cdot \Vert \tilde g(s)\Vert^2_{\mathcal H} ds <\rho^2.
\end{equation}

From Lemma \ref{lemma1} we know that there exists a local solution defined on a time interval $[0,\delta)$, with $\delta \le \tau$.
From \eqref{smallness_condition} and \eqref{rho}  we have that
\begin{equation}\label{stimah}
\begin{array}{l}
\displaystyle{ h(||A^\frac{1}{2}u_0(0)||_H)< h\left(\frac{\rho}{(1-\tilde{\mu})^\frac{1}{2}}\right)\leq h\left (  \frac 1 {2\bar{C}^\frac{1}{2}(N\tau)}h^{-1} \left (  \frac{1-\tilde{\mu}}{2}\right ) \right )< \frac{1-\tilde{\mu}}{2},}
\end{array}
\end{equation}
where we have used the fact that $\bar{C}(N\tau)>1.$
This implies, from  Lemma \ref{Lemma 2}, $E(0)>0$. Furthermore, from \eqref{assumptionPsi} and \eqref{stimah} we get
$$
\begin{array}{l}
\displaystyle{
E(0)\le \frac{1}{2}||u_1||_H^2+\frac{3}{4}(1-\tilde{\mu})||A^\frac{1}{2}u_0(0)||_H^2+\frac{1}{2}\int_{-\tau}^0 |k(s+\tau)|\cdot ||B^*u_t(s)||_H^2 ds}\\
\hspace{2,9 cm}
\displaystyle{ +\frac{1}{2}\int_0^{+\infty} \mu(s)||A^\frac{1}{2}\eta_0(s)||_H^2 ds <\rho^2,}
\end{array}
$$
which gives, recalling \eqref{rho},
\begin{equation}\label{september1}
\displaystyle{ h\left( \frac{2}{(1-\tilde{\mu})^\frac{1}{2}}\bar{C}^\frac{1}{2}(N\tau)E^\frac{1}{2}(0) \right) < h\left(\frac{2}{(1-\tilde{\mu})^\frac{1}{2}}\bar{C}^\frac{1}{2}(N\tau)\rho \right) \leq h\left (h^{-1}\left (  \frac {1-\tilde\mu} 2 \right)\right )=\frac{1-\tilde{\mu}}{2}.}
\end{equation}
Since $\bar{C}(N\tau)\geq \bar{C}(\tau)$,
then
\begin{equation}\label{J1}
 h\left( \frac{2}{(1-\tilde{\mu})^\frac{1}{2}}\bar{C}^\frac{1}{2}(\tau)E^\frac{1}{2}(0) \right)
\le  h\left( \frac{2}{(1-\tilde{\mu})^\frac{1}{2}}\bar{C}^\frac{1}{2}(N\tau)E^\frac{1}{2}(0) \right)<\frac{1-\tilde{\mu}}{2}.
\end{equation}
So,
we can apply Lemma \ref{Lemma 2} and we obtain
\begin{equation*}
\begin{array}{l}
\displaystyle{ E(t)>\frac{1}{4}||u_t(t)||_H^2+\frac{1-\tilde{\mu}}{4}||A^\frac{1}{2}u(t)||_H^2}\\
\hspace{1 cm}
\displaystyle{ +\frac{1}{4}\int_{t-\tau}^t |k(s+\tau)| \cdot ||B^*u_t(s)||_H^2 ds+\frac{1}{4}\int_0^{+\infty} \mu(s) ||A^\frac{1}{2}\eta^t (s)||_H^2 ds,}
\end{array}
\end{equation*}
for all $t\in [0, \delta ).$
In particular we have that
\begin{equation*}
E(t)>\frac{1}{4}||u_t(t)||_H^2, \quad \mbox{for}\  t\in [0, \delta ).
\end{equation*}
Therefore, we can apply
Lemma \ref{stimaE}, obtaining
$$
E(t)\leq \bar{C}(\tau) E(0)< \bar{C}(\tau)\rho^2,
$$
for any $t\in [0,\delta ]$. Since
\begin{equation}\label{step}
\begin{array}{l}
\displaystyle{ 0<  \frac{1}{4}||u_t(t)||_H^2+\frac{1-\tilde{\mu}}{4}||A^\frac{1}{2}u(t)||_H^2}\\
\displaystyle{+\frac{1}{4}\int_{t-\tau}^t |k(s+\tau)| \cdot ||B^*u_t(s)||_H^2 ds+\frac{1}{4}\int_0^{+\infty} \mu(s) ||A^\frac{1}{2}\eta^t (s)||_H^2 ds\leq E(t)\leq \bar{C}(\tau) E(0),}
\end{array}
\end{equation}
for all $t\in [0,\delta]$, then we can extend the solution to the entire interval $[0,\tau]$.

Now, observe that from \eqref{step} and \eqref{J1} we have
\begin{equation}
\label{stima a tau}
\begin{array}{l}
\displaystyle{h(||A^{\frac{1}{2}} u(\tau)||_H)\leq h\left(\frac{2}{(1-\tilde{\mu})^\frac{1}{2}}\bar{C}^\frac{1}{2}(\tau) E^\frac{1}{2}(0) \right)<\frac{1-\tilde{\mu}}{2}.}
\end{array}
\end{equation}
By continuity, \eqref{stima a tau}   implies that there exists $\delta '>0$ such that
$$
h(||A^\frac{1}{2}u(t)||_H)<\frac{1-\tilde{\mu}}{2}, \qquad \forall t \in [\tau,\tau+\delta ').
$$
From this, arguing as before, we deduce
$$
\begin{array}{l}
\displaystyle{ E(t)>\frac{1}{4}||u_t(t)||_H^2+\frac{1-\tilde{\mu}}{4}||A^\frac{1}{2}u(t)||_H^2+\frac{1}{4}\int_{t-\tau}^t |k(s+\tau)|\cdot ||B^*u_t(s)||_H^2 ds}\\
\hspace{2 cm}
\displaystyle{ +\frac{1}{4}\int_0^{+\infty} \mu(s) ||A^\frac{1}{2}\eta^t(s)||_H^2 ds,}
\end{array}$$
for any $t\in [\tau,\tau+\delta')$. In particular, also in such an interval we have $E(t)>\frac{1}{4}||u_t(t)||_H^2$. Hence, we can apply again Lemma \ref{stimaE} on the time interval $[0, \tau +\delta ')$ obtaining
$$
0<E(t)\leq \bar{C}(2\tau) E(0).$$
As before, we can then extend the solution the whole interval $[0,2\tau]$. At time $t=2\tau$ we have that
$$
\displaystyle{ h(\Vert A^{\frac 12} u(2\tau)\Vert_H)\le h\left(\frac{2}{(1-\tilde{\mu})^\frac{1}{2}}E^\frac{1}{2}(2\tau)\right) \le h\left(\frac{2}{(1-\tilde{\mu})^\frac{1}{2}}\bar{C}^\frac{1}{2}(2\tau)E^\frac{1}{2}(0)\right) <\frac{1-\tilde{\mu}}{2},}
$$
where we have used \eqref{J1}. Moreover, if $3\le N,$
$$
h\left(\frac{2}{(1-\tilde{\mu})^\frac{1}{2}}\bar{C}^\frac{1}{2}(3\tau)E^\frac{1}{2}(0)\right) \le h\left(\frac{2}{(1-\tilde{\mu})^\frac{1}{2}}\bar{C}^\frac{1}{2}(N\tau)E^\frac{1}{2}(0)\right)<\frac{1-\tilde{\mu}}{2}.
$$
Thus, one can repeat again the same argument.
By iteration, we then find a unique solution to the problem \eqref{forma_astratta2} on the interval $[0,
N\tau]$, where $N$ is the natural number fixed at the beginning of the proof.
Moreover, the definition \eqref{rho} of $\rho,$ ensures that
\begin{equation*}
h(\Vert A^{\frac 12} u(N\tau)\Vert_H)\le
h\left(\frac{2}{(1-\tilde{\mu})^\frac{1}{2}}E^\frac{1}{2}(N\tau)\right) \le  h \left (   \frac{2}{(1-\tilde{\mu})^\frac{1}{2}}\bar{C}^\frac{1}{2}(N\tau)E^\frac{1}{2}(0) \right )<\frac{1-\tilde{\mu}}{2},
\end{equation*}
where we have used \eqref{september1}.
Note that, by construction, \eqref{stima E dal basso} and \eqref{J2} are satisfied in the whole $[0, N\tau ).$ Then, from \eqref{J2},
$$
\begin{array}{l}
||U(t)||_\mathcal{H}^2\le 4E(t)\leq 4\bar{C}(N\tau)E(0)<4\bar{C}(N\tau)\rho^2
\end{array}
$$
and so
$$||U(t)||_\mathcal{H}\le 2\bar{C}^\frac{1}{2}(N\tau)\rho, \quad \forall\ t\in [0, N\tau].
$$
Thus,
we have proved that, under the assumption \eqref{condiz_gen} on the initial data, there exists a unique solution $U(\cdot)$ to problem \eqref{forma_astratta2} defined on the time interval $[0, N\tau ].$
Moreover,
\begin{equation*}
||U(t)||_\mathcal{H} \le C_\rho :=2\bar{C}^\frac{1}{2}(N\tau)\rho.
\end{equation*}
So far we have fixed
$\rho$  satisfying \eqref{rho}; now, eventually choosing a smaller $\rho,$ we assume  that $\rho $ satisfies the additional assumption
\begin{equation*}
L(C_\rho)=L(\bar{C}^\frac{1}{2}(N\tau)\rho )<\frac{\omega-\omega'}{2M}.
\end{equation*}
Then, the well--posedness assumption (W) of Theorem \ref{generaleCV} is satisfied on $[0, N\tau ].$ Therefore, we obtain that $U$ satisfies the exponential decay estimate \eqref{stimaesponenziale} and then
\begin{equation}\label{stimaN}
||U(t)||_{\mathcal H}\leq M\left( ||U(0)||_{\mathcal H}+\int_0^\tau |k(s)| e^{\omega s} ||\tilde g(s)||_{\mathcal H} ds \right) e^\gamma e^{-\frac{\omega -\omega'}{2} t}, \quad\forall\  t\in [0, N\tau ].
\end{equation}
In particular,
\begin{equation}\label{stimaN2}
||U(N\tau)||_{\mathcal H}\leq M\left( ||U(0)||_{\mathcal H}+\int_0^\tau e^{\omega s}|k(s)| \cdot ||\tilde g(s)||_{\mathcal H} ds \right) e^\gamma e^{-\frac{\omega -\omega'}{2} N\tau}.
\end{equation}
Now,  observe that
\begin{equation*}
\int_0^\tau e^{\omega s}|k(s)| \cdot ||\tilde g(s)||^2_{\mathcal H} ds \leq e^{\omega \tau}  \Biggl(\int_0^\tau |k(s)|ds \Biggr)^{1/2} \Biggl( \int_0^\tau |k(s)| \cdot ||\tilde g(s)||^2_{\mathcal H} ds\Biggr) ^{1/2}.
\end{equation*}
Hence,
\begin{equation*}
\displaystyle{ ||U(t)||_{\mathcal H} \leq M
\rho \left ( 1+ e^{\omega\tau }{C^*}^{1/2}  \right ) e^{\gamma} e^{-\frac{\omega -\omega'}{2} t}},\quad \forall \ t\in [0, N\tau],
\end{equation*}
and then
\begin{equation}\label{WW1}
\displaystyle{ ||U(t)||^2_{\mathcal H} \leq 2 M^2
\rho^2  \left ( 1+ e^{2\omega\tau }{C^*}  \right ) e^{2\gamma} e^{-(\omega -\omega') t}}, \quad \forall \ t\in [0, N\tau],
\end{equation}
where $C^*$ is the constant defined in \eqref{Cstar}.
From \eqref{WW1} we deduce
\begin{equation}\label{quasi}
\begin{array}{l}
\displaystyle{
\Vert U(N\tau )\Vert^2_{\mathcal H} +\int_{N\tau}^{(N+1)\tau} e^{\omega (s-N\tau ) }\vert k(s)\vert\cdot  \Vert B^*u_t(s-\tau )\Vert^2_{H} \, ds
\quad\quad }\\
\displaystyle{
\le 2 M^2
\rho^2  \left ( 1+ e^{2\omega\tau }{C^*}  \right ) e^{2\gamma} e^{-(\omega -\omega') N\tau}
+e^{\omega\tau} b^2 \int_{N\tau}^{(N+1)\tau} \vert k(s)\vert\cdot  \Vert U(s-\tau)\Vert^2_{\mathcal H}\, ds.
}
\end{array}
\end{equation}

Now, observe that, for $s\in [N\tau , (N+1)\tau ],$ it results $s-\tau \in [(N-1)\tau , N\tau ],$ then from \eqref{WW1} we deduce
$$ \Vert U(s-\tau  )\Vert^2_{\mathcal H}\le  2 M^2
\rho^2  \left ( 1+ e^{2\omega\tau }{C^*}  \right ) e^{2\gamma} e^{-(\omega -\omega') (N-1)\tau }, \quad \forall \ s\in [N\tau , (N+1)\tau ].$$
This last estimate, used in \eqref{quasi}, gives
\begin{equation}\label{quasi2}
\begin{array}{l}
\displaystyle{
\Vert U(N\tau )\Vert^2_{\mathcal H} +\int_{N\tau}^{(N+1)\tau} e^{\omega (s-N\tau ) }\vert k(s)\vert\cdot  \Vert B^*u_t(s-\tau )\Vert^2_{H} \, ds
\quad\quad }\\
\displaystyle{
\le 2 M^2
\rho^2  \left ( 1+ e^{2\omega\tau }{C^*}  \right ) e^{2\gamma} e^{-(\omega -\omega') (N-1)\tau}
\left ( 1 +e^{\omega\tau }b^2 C^*  \right ).
}
\end{array}
\end{equation}

From \eqref{quasi2} and \eqref{stimaN1}, we then  deduce
\begin{equation*}
||U(N\tau)||^2_{\mathcal H}+\int_{N\tau}^{(N+1)\tau} |k(s)| \cdot ||B^*u_t(s-\tau )||_H^2 ds <\rho^2.
\end{equation*}
Thus (cf. with \eqref{condiz_gen}), one can argue as before on the interval $[N\tau, 2 N\tau ]$ obtaining a solution on $[0, 2N\tau].$ Iterating this procedure we get a global solution satisfying
$$\Vert U(t)\Vert_{\mathcal H}<C_\rho.$$
Therefore, we have showed
that the problem \eqref{forma_astratta2} satisfies the well-posedness assumption (W) of Theorem \ref{generaleCV}.
\end{proof}

We have then proved that, for suitably small data, solutions to problem  \eqref{forma_astratta2} are globally defined and their energies satisfy an exponential decay estimate. Therefore, one can state the following theorem.
\begin{Theorem}
Let us consider \eqref{equazione_riscritta}. Then, there exists $\delta >0$ such that
if
\begin{equation}\label{ipotesi}
(1-\tilde{\mu}) ||A^{\frac{1}{2}}u_0||_H^2+||u_1||_H^2+\int_0^{+\infty} \mu(s) ||A^{\frac{1}{2}}\eta _0||_H^2 ds+\int_0^\tau |k(s)| \cdot ||Bg(s-\tau)||_H^2 ds <\delta,
\end{equation}
then the solution $u$ is globally defined and it satisfies
\begin{equation}
\label{tesi}
E(t)\leq Ce^{-\beta t},
\end{equation}
where $C$ is a constant depending only on the initial data and $\beta>0$.
\end{Theorem}
\begin{proof}
By a direct application of Theorem \ref{generaleCV} and Theorem \ref{teorema2.5} we have that there exist $K, \gamma >0$ such that
\begin{equation}
\label{cond1}
||U(t)||_\mathcal{H}^2 \leq Ke^{-\gamma t},
\end{equation}
for all $t\geq 0$, if the initial data are suitably small. Now, observe that there exists a constant $\overline{C}>0$ such that
\begin{equation}
\label{cond2}
\int_{t-\tau}^t |k(s)|\cdot ||B^* u_t(s)||_H^2 ds \leq \overline{C}C^* e^{-\gamma (t-\tau)}.
\end{equation}
Then, from \eqref{cond1} and \eqref{cond2} we obtain \eqref{tesi}.
\end{proof}

\section{Examples}
\subsection{The wave equation with memory and source term}
Let $\Omega$ be a non-empty bounded set in $\RR^n$, with boundary $\Gamma$ of class $C^2$. Moreover, let $\mathcal{O}\subset \Omega$ be a nonempty open subset of $\Omega$. We consider the following wave equation:
\begin{equation}
\label{memory+source}
\begin{array}{l}
\displaystyle{u_{tt}(x,t)-\Delta u(x,t)+\int_0^{+\infty} \mu(s)\Delta u(x,t-s) ds+k(t)\chi _{\mathcal{O}} u_t(x,t-\tau)}\\
\hspace{7 cm}
\displaystyle{=|u(x,t)|^\sigma u(x,t), \qquad \text{in} \ \Omega \times (0,+\infty),}\\
\displaystyle{ u(x,t)=0, \qquad \text{in} \ \Gamma \times (0,+\infty),}\\
\displaystyle{ u(x,t)=u_0(x,t) \qquad \text{in} \ \Omega \times (-\infty,0],}\\
\displaystyle{ u_t(x,0)=u_1(x), \qquad \text{in} \ \Omega,} \\
\displaystyle{ u_t(x,t)=g(x,t), \qquad \text{in} \ \Omega \times (-\tau,0],}
\end{array}
\end{equation}
where $\tau >0$ is the time-delay, $\mu :(0,+\infty)\rightarrow (0,+\infty)$ is a locally absolutely continuous memory kernel, which satisfies the assumptions (i)-(iv), $\sigma >0$ and the damping coefficient $k(\cdot)$ is a function in $L^1_{loc}([0,+\infty))$ satisfying \eqref{Cstar}.
Then, system \eqref{memory+source} falls in the form \eqref{equazione_generale} with $A=-\Delta$ and $D(A)=H^2(\Omega)\cap H^1_0(\Omega).$

Defining $\eta_s^t$ as in \eqref{eta}, we can rewrite the system \eqref{memory+source} in the following way:
\begin{equation}
\label{memory+source_riscritta}
\begin{array}{l}
\displaystyle{ u_{tt}(x,t)-(1-\tilde{\mu})\Delta u(x,t)-\int_0^{+\infty} \mu(s)\Delta \eta^t(x,s) ds+k(t)\chi _{\mathcal{O}} u_t(x,t-\tau)}\\
\hspace{7 cm}
\displaystyle{=|u(x,t)|^\sigma u(x,t), \quad \text{in} \ \Omega\times (0,+\infty),}\\
\displaystyle{\eta_t^t(x,s)=-\eta_s^t(x,s)+u_t(x,t), \qquad \text{in} \ \Omega\times (0,+\infty)\times (0,+\infty),}\\
\displaystyle{ u(x,t)=0, \qquad \text{in} \ \Gamma \times (0,+\infty),}\\
\displaystyle{ \eta^t(x,s)=0, \qquad \text{in} \ \Gamma \times (0,+\infty), \quad \text{for} \ t\geq 0,}\\
\displaystyle{ u(x,0)=u_0(x):=u_0(x,0), \qquad \text{in} \ \Omega,}\\
\displaystyle{ u_t(x,0)=u_1(x):=\frac{\partial u_0}{\partial t}(x,t)\Bigr| _{t=0}, \qquad \text{in} \ \Omega,}\\
\displaystyle{ \eta^0(x,s)=\eta_0(x,s):=u_0(x,0)-u_0(x,-s), \qquad \text{in} \ \Omega\times (0,+\infty),}\\
\displaystyle{ u_t(x,t)=g(x,t), \qquad \text{in} \ \Omega\times (-\tau,0).}

\end{array}
\end{equation}
As before, we introduce the Hilbert space $L^2_\mu ((0,+\infty);H^1_0(\Omega))$ endowed with the inner product
$$
\langle \phi , \psi\rangle _{L^2_\mu ((0,+\infty);H^1_0(\Omega))} := \int_{\Omega} \left( \int_0^{+\infty} \mu(s) \nabla \phi (x,s) \nabla \psi (x,s) dx \right) ds,
$$
and consider the Hilbert space
$$
\mathcal{H}=H_0^1(\Omega)\times L^2(\Omega) \times L^2_\mu ((0,+\infty);H^1_0(\Omega)),
$$
equipped with the inner product
\begin{equation*}
\left\langle
\left (
\begin{array}{l}
u\\
v\\
w
\end{array}
\right ),
\left (
\begin{array}{l}
\tilde u\\
\tilde v\\
\tilde w
\end{array}
\right )
\right\rangle_{\mathcal{H}}:=(1-\tilde{\mu}) \int_{\Omega} \nabla u \nabla \tilde u dx +\int_{\Omega} v\tilde v dx+ \int_{\Omega} \int_0^{+\infty} \mu(s) \nabla w \nabla \tilde w ds dx.
\end{equation*}
Setting $U=(u,u_t,\eta^t)$, we can rewrite \eqref{memory+source_riscritta} in the form \eqref{forma_astratta2}, where
$$
\mathcal{A} \begin{pmatrix}
u\\
v\\
w
\end{pmatrix}
=
\begin{pmatrix}
v\\
(1-\tilde{\mu})\Delta u+\int_0^{+\infty} \mu(s) \Delta w(s) ds \\
-w_s +v
\end{pmatrix},
$$
with domain
\begin{eqnarray*}
{D(A)}&=&\{  (u,v,w)\in H_0^1(\Omega)\times H_0^1(\Omega)\times L^2_\mu ((0,+\infty);H^1_0(\Omega)): \\
& & (1-\tilde{\mu})u +\int_0^{+\infty} \mu(s)w(s)ds \in H^2(\Omega)\cap H_0^1(\Omega), \ w_s\in L^2_\mu ((0,+\infty);H^1_0(\Omega))\} ,
\end{eqnarray*}
$\mathcal{B}(u,v,\eta^t)^T := (0,\chi _{\mathcal{O}} v,0)^T$ and $F(U(t))=(0,|u(t)|^{\sigma}u(t), 0)^T$. For any $u\in H_0^1(\Omega)$ consider the functional
$$
\psi (u):= \frac{1}{\sigma +2} \int_{\Omega} |u(x)|^{\sigma+2} dx.
$$
By Sobolev's embedding theorem, we know that if $0<\sigma< \frac{4}{n-2}$, then $\psi$ is well-defined, and G\^ateaux differentiable at any point $u\in H_0^1(\Omega)$, with G\^ateaux derivative given by
$$
D\psi (u)(v)=\int_{\Omega} |u(x)|^\sigma u(x)v(x)dx,
$$
for any $v\in H_0^1(\Omega)$. Moreover, as in \cite{ACS}, if $0<\sigma\le \frac{2}{n-2}$, then $\psi$ satisfies the assumptions (H1), (H2), (H3). Define the energy as follows:
$$
\begin{array}{l}
\displaystyle{E(t):= \frac{1}{2}\int_{\Omega} |u_t(x,t)|^2 dx +\frac{1-\tilde{\mu}}{2}\int_{\Omega} |\nabla u(x,t)|^2 dx -\psi(u(x,t))}\\
\hspace{2 cm}
\displaystyle{ +\frac{1}{2}\int_{t-\tau}^t \int_{\mathcal{O}} |k(s+\tau)|\cdot |u_t(x,s)|^2 dx ds +\frac{1}{2}\int_0^{+\infty} \mu (s)\int_{\Omega} |\nabla \eta^t(x,s)|^2 dx ds.}
\end{array}
$$
Theorem \ref{generaleCV} applies to this model giving well-posedness and exponential decay of the energy for suitably small initial data, provided that the condition \eqref{ipotesi2} on the time delay holds for every $t\geq 0$.

\subsection{The plate system with memory and source term}
Let $\Omega$ be a non-empty bounded set in $\RR^n$, with boundary $\Gamma$ of class $C^2$. Let us denote $\nu(x)$ the outward unit normal vector at any point $x\in\Gamma.$ Moreover, let $\mathcal{O}\subset \Omega$ be a nonempty open subset of $\Omega$. We consider the following viscoelastic plate system:
\begin{equation}
\label{memory+source+plate}
\begin{array}{l}
\displaystyle{u_{tt}(x,t)+\Delta^2 u(x,t)-\int_0^{+\infty} \mu(s)\Delta^2 u(x,t-s) ds+k(t)\chi _{\mathcal{O}} u_t(x,t-\tau)}\\
\hspace{7 cm}
\displaystyle{=|u(x,t)|^\sigma u(x,t), \qquad \text{in} \ \Omega \times (0,+\infty),}\\
\displaystyle{ u(x,t)=\frac {\partial u}{\partial\nu }(x,t)=0, \qquad \text{in} \ \Gamma \times (0,+\infty),}\\
\displaystyle{ u(x,t)=u_0(x,t) \qquad \text{in} \ \Omega \times (-\infty,0],}\\
\displaystyle{ u_t(x,0)=u_1(x), \qquad \text{in} \ \Omega,} \\
\displaystyle{ u_t(x,t)=g(x,t), \qquad \text{in} \ \Omega \times (-\tau,0],}
\end{array}
\end{equation}
where $\tau >0$ is the time-delay, $\mu :(0,+\infty)\rightarrow (0,+\infty)$ is a locally absolutely continuous memory kernel, which satisfies the assumptions (i)-(iv), $\sigma >0$ and the damping coefficient $k(\cdot)$ is a function in $L^1_{loc}([0,+\infty))$ satisfying \eqref{Cstar}.
This system again falls in \eqref{equazione_generale} for $A=\Delta^2$ with domain $D(A)= H^4(\Omega)\cap H^1_0(\Omega).$
Defining $\eta_s^t$ as in \eqref{eta}, we can rewrite the system \eqref{memory+source+plate} in the following way:
\begin{equation}
\label{memory+source_riscritta}
\begin{array}{l}
\displaystyle{ u_{tt}(x,t)+(1-\tilde{\mu})\Delta^2 u(x,t)+\int_0^{+\infty} \mu(s)\Delta^2 \eta^t(x,s) ds+k(t)\chi _{\mathcal{O}} u_t(x,t-\tau)}\\
\hspace{7 cm}
\displaystyle{=|u(x,t)|^\sigma u(x,t), \quad \text{in} \ \Omega\times (0,+\infty),}\\
\displaystyle{\eta_t^t(x,s)=-\eta_s^t(x,s)+u_t(x,t), \qquad \text{in} \ \Omega\times (0,+\infty)\times (0,+\infty),}\\
\displaystyle{ u(x,t)=\frac {\partial u}{\partial\nu }(x,t)= 0, \qquad \text{in} \ \Gamma \times (0,+\infty),}\\
\displaystyle{ \eta^t(x,s)=0, \qquad \text{in} \ \Gamma \times (0,+\infty), \quad \text{for} \ t\geq 0,}\\
\displaystyle{ u(x,0)=u_0(x):=u_0(x,0), \qquad \text{in} \ \Omega,}\\
\displaystyle{ u_t(x,0)=u_1(x):=\frac{\partial u_0}{\partial t}(x,t)\Bigr| _{t=0}, \qquad \text{in} \ \Omega,}\\
\displaystyle{ \eta^0(x,s)=\eta_0(x,s):=u_0(x,0)-u_0(x,-s), \qquad \text{in} \ \Omega\times (0,+\infty),}\\
\displaystyle{ u_t(x,t)=g(x,t), \qquad \text{in} \ \Omega\times (-\tau,0).}

\end{array}
\end{equation}
Then, arguing analogously to the previous example, one can recast \eqref{memory+source_riscritta} in the form \eqref{forma_astratta2}. Moreover, for $(n-4)\sigma\le 4$ (cf.  e.g. \cite{MK}) the nonlinear source satisfies the required assumptions.

Theorem \ref{generaleCV} applies then to this model giving well-posedness and exponential decay of the energy for suitably small initial data, provided that the condition \eqref{ipotesi2} on the time delay holds for every $t\geq 0$.

\end{document}